\documentclass{ifacconf}

\usepackage{graphicx} 
\usepackage{natbib} 
\usepackage{amsmath}
\usepackage{amssymb}
\usepackage{xcolor}

\renewcommand{\Re}{\mathrm{Re}\,}
\renewcommand{\Im}{\mathrm{Im}\,}

\newcommand{\R}{\mathbb{R}}

\newif\ifOmitExtra
\OmitExtratrue
\OmitExtrafalse 

\ifOmitExtra
  \newcommand{\gray}[2]{#2}
\selectfont
\else
  \newcommand{\gray}[2]{{\color{gray}{#1}}} 
  \newcommand{\red}[2]{{\color{red}{#1}}} 
  \usepackage[absolute,overlay]{textpos}
  \usepackage[normalem]{ulem}

\makeatletter
\let\old@ssect\@ssect 
\makeatother
\usepackage[colorlinks=true,citecolor=blue]{hyperref}
\makeatletter
\def\@ssect#1#2#3#4#5#6{%
\NR@gettitle{#6}
\old@ssect{#1}{#2}{#3}{#4}{#5}{#6}
}
\fi

\begin{document}
\gray{
\begin{textblock*}{1\textwidth}(1cm, 1.5cm) 
\begin{center}
\LARGE 
!! This is a slightly extended (\gray{by grayed text}{}) and better-readable version of~the paper, with some misprints fixed (\red{red text}{}),  presented at IFAC Congress 2026,
August 23-28, 2026 in Busan, Republic of Korea\\[-3mm]
{\small $\copyright$ 2026 Andrey Tremba. This work has been accepted to IFAC for publication under a Creative Commons Licence CC-BY-NC-ND}
\end{center}
\end{textblock*}
}{}

\begin{frontmatter}

\title{\mbox{Re-opening}~PID~Controller~Stability~Domain in 3D via Ruled Surface by~D-partition\gray{}{\vspace{-2mm}}}


\author[First]{Andrey A. Tremba}

\address[First]{Institute~of~Control~Sciences~RAS,~Moscow,~Russia~(\mbox{e-mail}:~atremba@ipu.ru).
\gray{}{\vspace{-5mm}}
}

\begin{abstract} 
All stabilizing PID controllers form a set in three-dimensional space.
A novel viewpoint to its boundary
as a ruled surface (or surfaces) being cut with 3D planes
is presented. The characterization, being not too new, contributes
to an understanding of the stability set as the whole,
instead of the classical view as a stack of 2D slices, say, on
the P-coefficient. The viewpoint gives clear insight on
the structure of the PID stability region, and, in particular,
splits its boundary into continuous parts. It is followed by
natural 2D unwrapping of the stability set boundary.
It also correctly handles pure imaginary zeros in transfer
function. A wireframe 3D visualization reveals the structure of the stability set.
The presentation is valid both for ideal and filtered PID controllers, as well as for time-delay systems and other linear systems. Finally, based on the viewpoint, a simple formula for stability (fragility) radius is provided.
\gray{}{\vspace{-4mm}}
\end{abstract}

\end{frontmatter}

\gray{}{\vspace{-2mm}}
\section{Introduction} \gray{}{\vspace{-2mm}}
A PID controller is described by a vector of three parameters, $k = (k_D, k_I, k_P)$. In the standard parallel form
the coefficients of an ideal PID controller for a continuous-time system are factors for integrating, proportional and differentiating terms: $k_D s + k_P + \frac{1}{s}k_I$.

The problem of describing the whole
set of stabilizing PID controllers for a given linear SISO plant is a well-known one,
and it is explored in detail, e.g., in \cite{hohen2009,schrodel-etal2015} and references therein.
However, its convenient representation in the three-dimensional space of parameters is non-trivial.
With few exceptions (as in \cite{gu-etal2022}), one of the parameters is fixed or gridded on, limiting the
description of the stability set to a plane or a stacked planes' image.

In each PI- or PD-plane (slice) with fixed $k_D$ or $k_I$, the stability set has curved boundary, while in ID-plane it is bounded by segments, \cite{ackermann-kaesbauer2003}.
The latter property follows from the structure\gray{\footnote{It also holds for linear systems with and without time-delay, also for non-ideal PID controllers and other cases, see Section~\ref{sec:non-ideal-pid}.}}{} of PID controller.
Due to the property, very nice and solid three-dimensional
representation of its boundary is revealed and exploited in the paper.
It follows from D-partition method, also known as Parameter Space Approach (PSA),
D-decomposition, Boundary Crossing technique, Guardian Mapping etc., \cite{neimark1948, ackermann-etal2002, gryazina-etal2008}.
\gray{The method has been rediscovered many times, as the idea behind it quite simple, originating from two- and one-parameter cases.}{} Let's remind it briefly.

Given a plant transfer function (TF), the first step is to form characteristic polynomial $G(s, k)$ of the closed-loop system. 
Stability of the system is determined by its roots, which all have to be in the left complex half-plane (be a Hurwitz polynomial). The second step is to find ``boundary'' values of parameters, so that a root
of characteristic polynomial is on boundary between stable and unstable roots, namely
imaginary axis.
The axis is parameterized by $jw, w \in (-\infty, +\infty)$, and the boundary crossing condition is defined by the set:
\vspace{-0.5mm}
\begin{equation} \label{eq:main0}
\vspace{-0.5mm}
\{(k_D, \!k_I, \!k_P) : G(jw, k_D, \!k_I, \! k_P) \!= \! 0, \, w\! \in\! (-\infty, +\infty)\}.
\end{equation}
The set splits parameter 3D space into connected 3D regions. If PID controller parameters
``move'' within each region, the number of stable (and unstable) roots of the polynomials remains unchanged.
\gray{This follows from continuous dependence of the characteristic polynomial's roots on the parameters. }{}The number of stable/unstable roots may change only if at least one root becomes zero or pure imaginary, with \eqref{eq:main0} holds.

Another case is \gray{continuity violation, happening in }{}\emph{degree drop condition},
also known as IRB (Infinite Root Boundary) in PSA. It may be thought as condition for ``inifinite unstable'' root. Here $\deg G$ is defined as maximum degree of term $s$ among all $k$:
\vspace{-1mm}
\begin{equation} \label{eq:main0-degree-drop}
\vspace{-0.5mm}
\{(k_D, \!k_I, \!k_P) : G_{\deg G}(k_D, \!k_I, \! k_P) \!= \! 0\}.
\end{equation}
Full D-partition is given by sets \eqref{eq:main0} and \eqref{eq:main0-degree-drop}, partitioning parameter space into regions.
The last step is to recover the stability set by selecting regions with $\deg G$ stable roots.

Having $k_P$ fixed, the D-partition in DI-slice is presented by a set of so-called \emph{critical} (or singular) lines. This characterization allows to state
a problem of determining intervals of $k_P$ coefficient, providing
different line configuration. Recently, it was successfully solved by \cite{hohen2009} and by \cite{hwang-etal2024}.
Having the stable $k_P$ intervals, the whole PID stability region can be easily recovered by gridding within the intervals, see details in Section~\ref{sec:DI-plane}.

The goal of this note is \gray{not to provide completely new solution of the problem of finding the stability region, but rather }{}to present another viewpoint onto D-partition in 3D as a parameterized \emph{ruled surface}.
The key difference is that the main parameter is not a controller coefficient, but the ``native''
frequency parameter $w$.
\gray{While following almost the same ways, t}{T}he viewpoint allows to get new results for the case of pure imaginary zeros of TF, easy calculation of fragility radius, etc. 
The main contribution list of this paper is:
\begin{itemize}
\item
a viewpoint to the boundary of stability region as a specific ruled surface intersected by planes,
\item
a viewpoint to the stability region boundary as parametric family of segments,
\item
the stability set boundary as set of smooth parts,
\item
uniform treatment of systems with pure imaginary zeros of transfer function (roots of its numerator),
\item
sufficient condition for localizing stability set (and stable intervals of $k_P$) by D-partition in PD-plane,
\item
new, explicit formula for the distance to the closest unstable PID controller (stability/fragility radius).
\end{itemize}

Except for the last item, the results are applicable to time-delay systems and non-ideal (filtered) PID controllers.
The viewpoint gives clear understanding of the stability peak condition and some of critical $k_P$,
provided in
\cite{bajcinca2006, soylemez-munro-baki2003, hwang-etal2024}.

\gray{}{\vspace{-2mm}}
\section{D-partition for PID controller in 3D} \gray{}{\vspace{-2mm}}
Consider a linear plant with transfer function $W(s)$: 
\[ \gray{}{\textstyle}
W(s) = \frac{N(s)}{D(s)} = \frac{a_m s^{\red{m}{n}} + ... + a_1 s + a_0}{b_n s^n + ... + b_1 s + b_0}.
\]
Ideal PID controller has transfer function $C(s) = s k_D + k_p + \frac{1}{s}k_I$. The closed-loop system has transfer function $W_{cl}(s) = \frac{C(s)W(s)}{1 + C(s)W(s)}$
with characteristic polynomial
\begin{equation} \label{eq:charpoly}
G(s, k) = (s^2 k_D + s k_P + k_I) N(s) + s D(s).
\end{equation}
The only assumption on the transfer function is that $W(s)$ is proper or strictly proper, e.g. $\deg N \leq \deg D$, with $N(s), D(s)$ having no common roots\gray{\footnote{A controlled plant's TF is to be a strictly proper one because of ideal differentiation term in PID. For filtered PID see Section~\ref{sec:non-ideal-pid}.}}{}.
Typically, cases without pure imaginary roots in $N(s)$ are considered, but we do not put such restriction.

For the following, it is convenient to sort axes as $k = (k_D, k_I, k_P)$, presenting PID controllers
in ``DIP-space'', with vertical Z-axis being $k_P$ axis.

\gray{}{\vspace{-2mm}}
\subsection{Derivation of parametric boundary equations} \gray{}{\vspace{-2mm}}
As all plant coefficients are real, \gray{from symmetry }{}the main D-partition equation \eqref{eq:main0} for \eqref{eq:charpoly}
is enough to solve on $[0, +\infty)$, as:
\begin{equation} \label{eq:main-jw}
(-w^2 k_D + j w k_P + k_I) N(j w) + jw D(jw) = 0.
\end{equation}
The key point of D-partition theory is that \eqref{eq:main-jw} is a \emph{complex-valued} equation, \gray{effectively }{}being a \emph{system of two real} ones:
\begin{equation} \label{eq:main-system}
\!\left\{\!\!
\begin{array}{l}
(-w^2 k_D\! +\! k_I) N_e\! -\! w^2 k_P N_o\! -\! w^2 D_o\! =\! 0, \\
(-w^2 k_D\! +\! k_I) w N_o\! +\! w k_P N_e\! +\! w D_e\! =\! 0. \\
\end{array}
\right.
w \in [0, \infty),
\end{equation}
Here even-odd polynomial decomposition is used, essentially responsible for real and imaginary parts\footnote{In some papers, it is denoted as $N(jw) = N_e(w^2) + j w N_o(w^2)$.}, with arguments ``$(-w^2)$'' omitted for brevity hereafter.
\[
\begin{array}{c}
\!N(s)\! =\! N_e(s^2)\! +\! s N_o(s^2) \rightarrow N(jw)\! =\! N_e(-w^2)\! +\! j w N_o(-w^2) \\
\!D(s)\! =\! D_e(s^2)\! +\! s D_o(s^2) \rightarrow D(jw)\! =\! D_e(-w^2)\! +\! j w D_o(-w^2)
\end{array}
\]

Equation \eqref{eq:main-jw} is also known in another form:
\[
\gray{}{\textstyle}
-w^2 k_D + j w k_P + k_I = - \frac{jw D(jw)}{N(jw)} = - jw W^{-1}(jw).
\]
It follows from dividing \eqref{eq:main-jw} on $N(jw)$ \emph{prior to the splitting} into real and imaginary parts:
\begin{equation} \label{eq:main-system-traditional}
\!\!\left\{\!\!
\begin{array}{l}
-w^2 k_D + k_I =\! -\Re(jw W^{-1}(jw)) = w \Im W^{-1}(jw), \\
w k_P = -\Im(jw W^{-1}(jw)) = -w \Re W^{-1}(jw).
\end{array}\!\!\!
\right.
\end{equation}
While \eqref{eq:main-system-traditional} is inapplicable directly to cases with TF having pure imaginary zeros (imaginary roots of $N$), it is very convenient for exploration. A solution for the pure imaginary zeros case is presented in Section \ref{sec:imaginary-zeros}.

The case of $w = 0$ is usually treated separately, splitting \eqref{eq:main-system-traditional}
into Real Root Boundary (RRB) condition at $w = 0$:
\begin{equation} \label{eq:rrb-3d}
k_I = 0,
\end{equation}
and Complex Root Boundary (CRB) for $w > 0$, leading to the parametric on $w$ solution:
\begin{equation} \label{eq:crb-3d}
\!\left\{
\begin{array}{cl}
\!\!k_I \!-\! w^2 k_D &\!= \gray{}{\textstyle}
w \Im W^{-1}(jw) = w^2\frac{D_o N_e - D_e N_o}{N_e^2 + w^2 N_o^2},\\
\gray{}{\textstyle}
\!\!k_P &\!= \gray{}{\textstyle}
-\Re W^{-1}(jw) = -\frac{D_e N_e + w^2 D_o N_o}{N_e^2 + w^2 N_o^2}.\\
\end{array}
\right.
\end{equation}
In generic 3-parametric controller case, the two CRB equations \eqref{eq:crb-3d} (or \eqref{eq:main-system}) on three variables would result in a surface in 3D.
Fortunately, ideal PID controller reveals $-w^2 k_D + k_I$ term as whole\footnote{There exists a smart parameter change workaround for non-ideal PID controller, see Section~\ref{sec:non-ideal-pid}.},
\cite{ackermann-kaesbauer2003}.
Thus for fixed $w$ CRB \eqref{eq:crb-3d} describes \emph{a line} on
the horizontal DI-slice with $k_P = -\Re W^{-1}(jw)$, with the slope $w^2$ within the slice,
\cite{soylemez-munro-baki2003, morarescu-etal2011}, etc.
This brilliant and concise geometric description is the core of all results for finding stability set
of PID controllers, see Section \ref{sec:DI-plane} for details.

Finally, IRB \eqref{eq:main0-degree-drop} depends on degrees of polynomials $N$ and $D$.
If\footnote{For a strictly proper plant it means $\deg N = \deg D - 1$, ($m = n-1$).} $\deg N \geq \deg D - 1$, then IRB has solution
\gray{}{\vspace{-1mm}}
\begin{equation} \label{eq:irb-3d}
\gray{}{\textstyle}
k_D = k_{D\infty} \doteq -\frac{b_{m+1}}{a_m}.
\gray{}{\vspace{-1mm}}
\end{equation}
where $b_{m+1}$ is the denominator coefficient at $s^{m+1}$, and it equals to zero if $D$ has no such term  (i.e. $\deg N = \deg D$).
It~directly follows from leading term of \eqref{eq:charpoly} being $s^2 k_D a_m + s b_{m+1}$.
Otherwise, there is no degree drop for any $k$, and IRB adds no boundary to D-partition.
If present, let's call the plane \eqref{eq:irb-3d} as the second critical plane.

The boundary formulas \eqref{eq:rrb-3d}, \eqref{eq:crb-3d}, \eqref{eq:irb-3d} for RRB, CRB and IRB are well-known and successfully used to explore
full set of stabilizing PID controllers, e.g. \cite{hwang-etal2024} and its references, see also Section~\ref{sec:DI-plane}.
Let's look at the base
equation \eqref{eq:main-system} closely, recovering the missing case of plant TF with imaginary zeros.

\gray{}{\vspace{-2mm}}
\subsection{Plants with pure imaginary zeros in TF} \gray{}{\vspace{-2mm}}
\label{sec:imaginary-zeros}
For a fixed $w$, main D-partition system of equations \eqref{eq:main-system} is a linear system of
equations, with respect to the terms $X = -w^2 k_D + k_I$ and $Z = k_P$.
It has explicit solution almost everywhere, \cite{gryazina-etal2008}:
\gray{}{\vspace{-2mm}}
\begin{align} \nonumber
\!\!\!\begin{bmatrix}
\! X\! \\
\!Z\!
\end{bmatrix} \!= \!
\begin{bmatrix}
-w^2 k_D + k_I \\
k_P
\end{bmatrix}
\!= \!
\begin{bmatrix}
N_e & -w^2 N_o \\
w N_o & w N_e \\
\end{bmatrix}^{-1}
\begin{bmatrix}
w^2 D_o \\
-w D_e
\end{bmatrix} \!=\!\! \\
\label{eq:crb-complete}
= \frac{1}{w (N_e^2 + w^2 N_o^2)}
\begin{bmatrix}
w N_e & w^2 N_o \\
-w N_o & N_e \\
\end{bmatrix}
\begin{bmatrix}
w^2 D_o \\
-w D_e
\end{bmatrix} = \\  \nonumber
= \frac{1}{N_e^2 + w^2 N_o^2}
\begin{bmatrix}
w^2(D_o N_e - D_e N_o) \\
-(w^2 D_o N_o + D_e N_e) \\
\end{bmatrix} = \\  \nonumber
w \neq 0, \quad N_e^2 + w^2 N_o^2 \neq 0. \quad \quad \quad \quad \quad  \phantom{ }
\gray{}{\vspace{-2mm}}
\end{align}
Of course, it coincides with \eqref{eq:main-system-traditional}, or
\eqref{eq:crb-3d},
but explicitly reveals
\emph{critical frequences}.

The common factor $w$ in \eqref{eq:main-system} is left intentionally.
The critical frequency\gray{\footnote{The stability equation $s = jw$ with for any characteristic polynomial with real coefficients always have the solution $w = 0$ as critical, due to vanishing imaginary part.}}{} $w = 0$ results in solution of main system
\eqref{eq:main-system} being the very same RRB condition \eqref{eq:rrb-3d}.
Let's call it the first critical hyperplane. In DIP-space it is a vertical plane,
coinciding with PD-plane with $k_I = 0$.

The missing case of $N_e^2(-w^2) + w^2 N_o^2(-w^2) = |N(jw)|^2 = 0$ in denominator of \eqref{eq:crb-complete} stands exactly for
pure imaginary zeros $w_i$ of plant TF: $N(j w_i) = 0 = W(jw_i)$.
Likewise for D-partition in 2D-space, cf. \cite{gryazina-etal2008}, the vanishing at $w_i$ determinant of the matrix
\gray{\[
\begin{bmatrix}
N_e & -w^2 N_o \\
w N_o & w N_e \\
\end{bmatrix}
\]}{$\begin{bmatrix}
N_e & -w^2 N_o \\
w N_o & w N_e \\
\end{bmatrix}$}
makes two equations in \eqref{eq:main-system} be linearly dependent, in terms of linear equation system on $X$ and $Z$.
In PID controller case, however, the pure imaginary roots nullify the matrix completely at $w_i$.
The resulting equation \eqref{eq:main-jw} becomes $jw_i D(jw_i) = 0$. It has no solution for
$w_i > 0$, as $N, D$ have no common roots. Therefore $w = w_i$ do not add extra boundary conditions.
The last remaining case of zero root $s_i = w_i = 0$ of $N$ stands for unstabilizable (or marginally stable) system for any $k$\footnote{It is said that the closed-loop system TF has catastrophic zero root cancellation then.
}.

\gray{}{\vspace{-2mm}}
\subsection{Simplifications and stability set estimates} \gray{}{\vspace{-2mm}}
Having \eqref{eq:crb-complete} in mind, the full boundary equations are:
\begin{equation} \label{eq:full-ddec-3d}
\begin{array}{lc}
k_I = 0, \quad \quad \quad \quad \quad \quad \quad \quad \text{(RRB)} \\
k_D = k_{D\infty}, \quad \quad \quad \text{(IRB, if $\deg N \geq \deg D - 1$)} \\
\left\{
\begin{array}{clc}
\!\!k_P &\!=\! -\Re W^{-1}(jw), & \text{(updated CRB)}\\
\!\!w k_D \!-\! \frac{1}{w} k_I &\!=\! -\Im W^{-1}(jw), & w > 0, N(j w) \neq 0.\\
\end{array}
\right.
\end{array}
\end{equation}
The description slightly differs from one in \cite{hwang-etal2024} and other papers, by a) patching CRB, and b) writing down the last equation in symmetric form.
It consists of one (or two) vertical plane(s) in DIP-space, and a surface.
At frequencies $w_i : N(j w_i) = 0$ the surface has infinite discontinuity, without
sign change.

Formula \eqref{eq:full-ddec-3d} immediately reveals connection with inverse of plant TF, with
straightforward numeric evaluation.
As it said above, the updated CRB equation is valid for arbitrary linear system, e.g. for systems with time-delay
\gray{}{\vspace{-2mm}}
\[
\gray{}{\textstyle}
W(s) = \frac{N(s)}{D(s)} e^{-\tau s}.
\]
The PID stability domain in this form was studied in \cite{morarescu-etal2011} etc.
It is applicable to fractional-order systems governed by classic PID controller as well, also
with non-polynomial TF.

Another simplification comes from even-odd decomposition.
All polynomial terms do depend on \emph{even} degrees of $w$, allowing substitution
$u = w^2$, twice decreasing degree.
As the result, rational functions in RHS in \eqref{eq:crb-3d} have degree of order
$n$ on $u$, despite polynomial multiplication. It is important for numeric calculation of extremal points
of \emph{generator of singular frequences} (e.g. see \cite{bajcinca2006, hwang-etal2024}), used
in \eqref{eq:crb-3d}, \eqref{eq:full-ddec-3d}:
\gray{}{\vspace{-1.5mm}}
\begin{equation} \label{eq:kp-generator}
\gray{}{\textstyle}
k_P(w) = -\Re W^{-1}(jw) = -\frac{D_e N_e + w^2 D_o N_o}{N_e^2 + w^2 N_o^2}.
\end{equation}
We conclude the section with few notes about limiting the stability set.
First, a stabilizing controller should have $k_I \geq 0$ (if $a_0 \neq 0$), $k_D > k_{D\infty}$ (if present), \cite{soylemez-munro-baki2003}.
In fact, a reasonable large box for $k_P$ may be chosen in advance as well.
A bounded exploration area for stabilizing PID controller helps avoiding
infinite branches of bounding surface, and successfully treat infinite number of branches in
time-delay systems.

\gray{}{\vspace{-2mm}}
\section{CRB as ruled surface(s)} \gray{}{\vspace{-2mm}}
We remind that for fixed $w$, CRB \eqref{eq:crb-3d} or \eqref{eq:full-ddec-3d} describe a
line in DI-plane (a horizontal slice of DIP-space, also named as DI-slice), with constant $k_P(w)$
by \eqref{eq:kp-generator}.

A surface is called \emph{ruled} if it can be presented in form of \emph{base curve} (directrix)
$A(w)$ and \emph{director curve} (vector) $B(w)$, \cite{kuhnel2005}:
\gray{}{\vspace{-1mm}}
\begin{equation} \label{eq:ruled-surface-wt}
L(w, t) = A(w) + t B(w).
\end{equation}
For fixed $w$ it is a line, called \emph{ruling}. Informally speaking, the surface is formed by ``moving'' the ruling in 3D.
This is exactly the case for PID controller's CRB \eqref{eq:full-ddec-3d}.
Parameter $t$ stands for a line parameterization in DI-slice for fixed $k_P(w)$, brought by frequency parameter $w$.

From \eqref{eq:full-ddec-3d}, the base curve is continuous for $w \neq w_i$, where $w_i$ accounts for pure imaginary zeros $j w_i$ of plant TF. We remind that axes in DIP-space are ordered as $(k_D, k_I, k_P)$. The base and director curves can be chosen as:
\begin{equation}
\label{eq:ruled-3d}
A(w) = \begin{bmatrix}
k_{D0}(w)\\
k_{I0}(w) \\
-\Re W^{-1}(jw)
\end{bmatrix}
= \begin{bmatrix}
\textstyle -\frac{1}{w}\Im W^{-1}(jw)\\
 0 \\
-\Re W^{-1}(jw)
\end{bmatrix}
\!\!, \\
\end{equation}
\begin{equation}
\gray{}{\textstyle}
B(w) = \begin{bmatrix}
1/w \gray{}{\vphantom{k_{D0}(w)}}\\
w \gray{}{\vphantom{k_{I0}(w)}} \\
0 \gray{}{\vphantom{-\Re W^{-1}(jw)}}
\end{bmatrix}\!\!, \\ \nonumber
\quad w > 0, \; w \neq w_i.
\gray{}{\vspace{-1mm}}
\end{equation}
First two terms in each vector stand for line representation in a DI-slice, while third component is
the very same generator of singular frequencies $k_P(w)$ \eqref{eq:kp-generator}, selecting the horizontal slice.

The representation is not unique. The third component is subject to re-parametrization, e.g. $u = w^2$.
But more freedom is in first two components, defining ruling lines.
For example, from \eqref{eq:crb-3d} the line parameterization may be $B(w) = [1, w^2, 0]^T$, up to any factor (or function on $w$, being nonzero and continuous).
For a fixed $w$, the \emph{base point} $[k_{D0}, k_{I0}, ...]^T$ may be chosen on the line in arbitrary way, being continuous on $w$ only.
It has to satisfy line equation
\gray{}{\vspace{-1mm}}
\begin{equation} \label{eq:line-abc}
\gray{}{\textstyle}
w k_D - \frac{1}{w} k_I + c(w) = 0, \quad w > 0, w \neq w_i.
\end{equation}
Here $c(w)$ denotes a shortcut for imaginary part of inverse TF, from \eqref{eq:crb-3d}:
\gray{}{\vspace{-1mm}}
\begin{equation} \nonumber
\gray{}{\textstyle}
c(w) \doteq \Im W^{-1}(j w) = w \frac{D_o N_e - D_e N_o}{N_e^2 + w^2 N_o^2}.
\end{equation}
The base curve $A(w)$ in \eqref{eq:ruled-3d} is a flat curve in the DP-plane. The choice is natural due to RRB condition $k_I = 0$ and necessary stability condition $k_I > 0$, see Section \ref{sec:PD-plane}, also \cite{morarescu-etal2011}.

Geometrically, the surface is formed by a line moving up and down in DIP-space, with constantly increasing slope in a DI-slice, see Fig.~\ref{fig:schema}.
For fixed $w$, \eqref{eq:ruled-surface-wt} expresses the ruling as a line originating from a base point, which move in PD-plane as $w$ increases. The line has natural parameterization by $t \in \R$.

\gray{}{\vspace{-4mm}}
\begin{figure}[!h]
\begin{center}
\gray{\includegraphics[width=4.3cm]{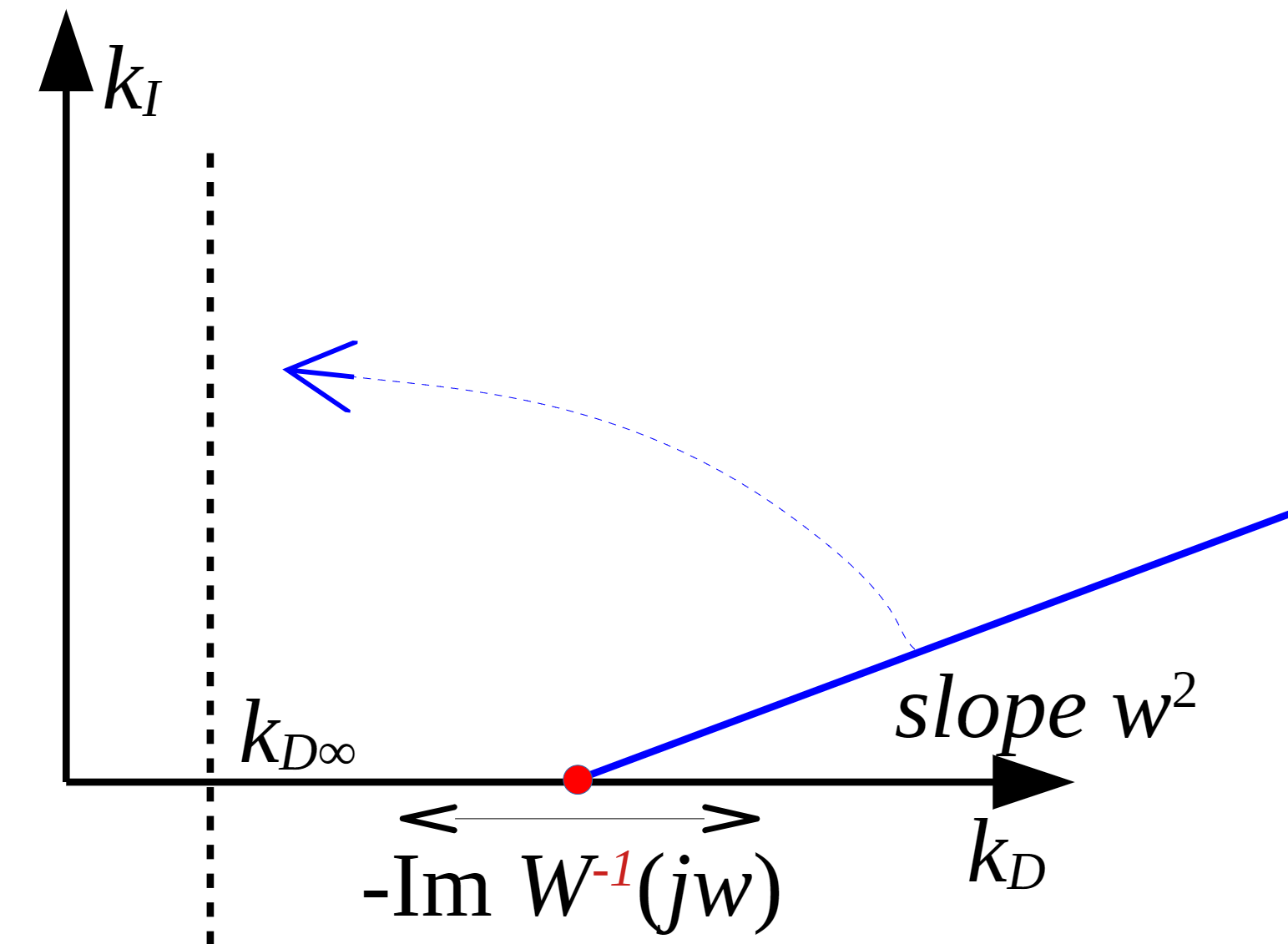}}{\includegraphics[width=4.3cm]{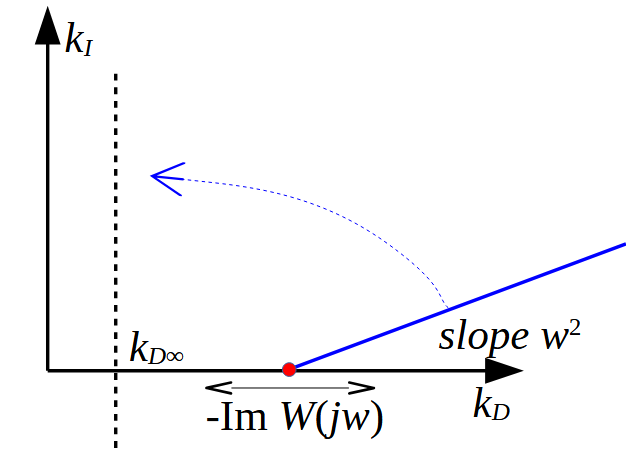}}~%
\includegraphics[width=4.3cm]{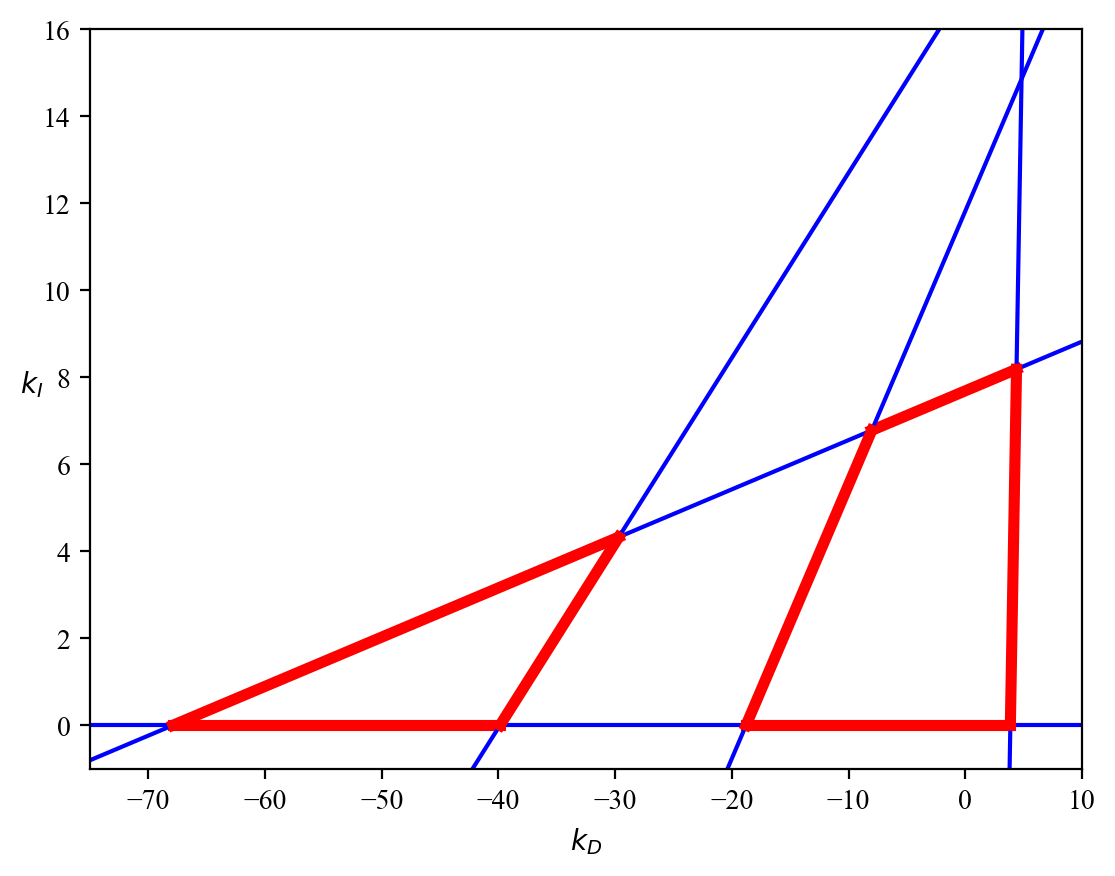}
\label{fig:ex3-slice}
\gray{}{\vspace{-4mm}}
\caption{Evolution of the ruling in projection to DI-plane with increasing $w$, in top-down view in the direction opposite to the Z axis (left). A DI-slice with stable components for $k_P = -3$ for Example 2 (right).}
\label{fig:schema}
\end{center}
\end{figure}

\begin{thm}
D-partition for PID controller is made by one (or two orthogonal) planes and
$m+1$ ruled surfaces where $m$ is a number of distinct pure imaginary zeros of plant TF.
Complex conjugate zeros $j w$ and $-j w$ are counted as one.
\gray{}{\vspace{-1mm}}
\end{thm}
\begin{pf}
From \eqref{eq:full-ddec-3d}, the hyperplanes are explicitly described by RRB and IRB (under degree condition), while CRB falls apart by distinct pure imaginary zeros $w_i > 0 : W(jw_i) = N(j w_i) = 0, i = 1, ..., m$ into $m+1$
ruled surfaces. Thought the surfaces are described by the same equation \eqref{eq:full-ddec-3d}, they have different domains
$(0, w_1), (w_1, w_2), ..., (w_m, \infty)$.
\end{pf}
As a corollary, if plant TF has no pure imaginary zeros, then a sole ruled surface present.

In any case, let's call the CRB in 3D as the ruled surface, even if consists of few continuous parts.
The ruled surface has multiple self-intersections, and it intersects both critical planes at rational curves.

At pure imaginary zeros, the surface goes to infinity and back alongside plane
\gray{}{\vspace{-1mm}}
\[
w_i k_D - \frac{1}{w_i} k_I - w_i k_P \lim_{w \to w_i}\frac{D_o N_e - D_e N_o}{D_e N_e + w^2 D_o N_o} = 0.
\]
The limit at $k_P$ coefficient do exist. It can be proved by cancellation of vanishing $(w^2 - w^2_i)$-based terms of $N_e$ and $N_o$ in numerator and denominator polynomials, with non-vanishing $D$.

Symmetric form on $w$ in \eqref{eq:line-abc} highlights surface rotation in DI-projection. For the small frequencies $w$ the surface is almost parallel and close to the vertical PD-plane, originating from infinite $k_P$ as $w \to 0$, with the sign of $b_0/a_0$ ratio.
After some oscillations with gradual counterclockwise rotation (if look from upside-down along Z-axis to the DI-plane, cf. Fig.~\ref{fig:schema}), it becomes almost parallel to the vertical PI-plane.
It approaches to the second critical plane $p_D = p_{D\infty}$, if it is present, or goes to infinity.
For time-delay systems, it undergo infinite number of upside-down crossing of $k_D = 0$ plane, but each is the crossing is almost vertical starting from some $w$.

\gray{}{\vspace{-2mm}}
\subsection{Polygonal splitting in DI-slice for fixed \texorpdfstring{$k_P$}{k P}} \gray{}{\vspace{-2mm}}
\label{sec:DI-plane}
Let's fix proportional gain to some value $k_{P0}$, and consider the DI-slice $k_P = k_{P0}$ as 2D parameter space.
\gray{The first and the second (if present) critical planes intersect it on lines $k_I = 0$ and $k_D = k_{D\infty}$.}{}
For $w$ progressing from $0$ to $\infty$, the ``going up-and-down'' ruled surface intersects the slice at
\[
w_\ell : k_P(w_\ell) = k_{P0},
\]
giving function \eqref{eq:kp-generator} ``generator'' name, \cite{bajcinca2006}, also used in
\cite{ackermann-kaesbauer2003}.
Each intersection adds boundary line to the slice, with equation
\gray{}{\vspace{-1mm}}
\begin{equation} \label{eq:critical-line}
\gray{}{\textstyle}
w_\ell k_D - \frac{1}{w_\ell} k_I + c(w_\ell) = 0, \quad \ell = 1, ...
\gray{}{\vspace{-1mm}}
\end{equation}
For a plant delay free TF, there is fixed number of lines, as \eqref{eq:kp-generator} is rational on $w$. For time-delay case there is infinite number of lines, but these may be grouped for large $w$, \cite{hohen2009}.

Intersection of the DI-slice with the first and the second critical planes add
special lines $k_I = 0$ and $k_D = k_{D\infty}$ (if the latter is present).
In total, the DI-slice is divided to convex polyhedra, as in Fig.~\ref{fig:ex3-slice}, making a D-partition, \cite{datta-etal2000}.
Exploring lines configuration dependence on $k_{P0}$, stable intervals over $k_P$ can be found by search of critical $k_P$ values with emerging or disappearing of polygons or polyhedra. As an example,
triple crossing lines \eqref{eq:critical-line} in the plane means a point of triple self-intersection of the ruled surface.

In \cite{hohen2009} and in \cite{hwang-etal2024}, the stable $k_P$ intervals problem is completely solved by enumerating critical values of $k_P$, eventually solving the whole problem of describing all stabilizing PID controllers.


Returning to the planar DI-slice with fixed $k_{P0}$, there are algorithms to select stable polygons in $(k_D, k_I)$ plane. The stable polygons are encountered either by considering different signs combinations (based on Hermite-Biehler stability criterion), \cite{datta-etal2000}, or by ``hatching'' (shading) technique, dating back to \cite{neimark1948}, also cf. \cite{soylemez-munro-baki2003}. In 2D stability analysis, the shading technique indicates\footnote{Or, inversely, unstable side, as in \cite{hohen2009}.} the ``more stable'' side of a boundary line, by analysis of polynomial roots coming from the unstable part of complex plane to the stable part.
For Hurwitz case it means that if a parameter is crossing boundary line to the proper (shaded) side, then a polynomial root is crossing imaginary line from right to the left, increasing number of stable roots.
Analytic condition of detecting the shaded sides for PID controller may be found e.g. in \cite{bajcinca2006}.

In \cite{hohen2009}, the shading condition is simplified by analysis of the generator \eqref{eq:kp-generator}. Namely, if generator $k_P(w)$ crosses line $k_P = k_{P0}$ in ``up'' direction,
then upper (left) side of the director line in $k_D-k_I$ plane is shaded. Alternatively, if $k_P(w)$ crosses the line in ``down'' direction, then lower (right) side of the line is shaded.

The shading is arising naturally, if one side of the ruled surface is selected as ``internal'' (``more stable'') one. Then stability region belongs to the ``internal'' side with respect to all parts of the surface, similar to 2D case \cite{neimark1948}. It is also noted in \cite{morarescu-etal2011}.

\gray{}{\vspace{-2mm}}
\subsection{Base point trajectory in DP-plane} \gray{}{\vspace{-2mm}}
\label{sec:PD-plane}
PD-controller (PID without integration term) is almost never used in practice, as it is vulnerable to cumulative errors.
Meanwhile, taking $k_I = 0$ is useful for exploring geometry of stability set boundary.
As said, in many examples it is a part of the stability set boundary.

The base curve following the base point \eqref{eq:ruled-3d} is on the DP-plane. It is described by scaled
inverse plant transfer function, rational in case of delay free plant:
\gray{}{\vspace{-1mm}}
\begin{equation} \label{eq:dp-plane-ddec}
\gray{}{\textstyle}
k_P(w) \!=\! -\Re W^{-1}(jw), \;\;
k_D(w) \!=\! -\frac{1}{w}\Im W^{-1}(jw).
\gray{}{\vspace{-1mm}}
\end{equation}
Alongside critical line $k_D = k_{D\infty}$ (if the second critical plane is present), the curve
\eqref{eq:dp-plane-ddec} represents D-partition on the DP-plane for the characteristic polynomial without extra $s$ factor: $G(s, k) = (s k_D + k_P) N(s) + D(s)$.
While the slice $k_I = 0$ represents marginally stable PID controller,
it may stand for a stabilizing PD-controller. If such a stabilizing controller exists, then adding small enough $0 < k_I << 1$ makes a stabilizing PID controller.

Thus stability regions in DP-space, rarely used for analysis, represent simple sufficient condition for a stabilizing PID controllers, not being a necessary though.
It allows to describe a part of PID stability domain, ``visible'' from $k_I=0$ plane, by selecting $(0, t_1(w))$ intervals on the line \eqref{eq:line-pd}, for the $w \in W$, where $W$ stands for boundary of  stability region in DP-plane, see Section~\ref{sec:split} for definition of stable intervals.

\gray{}{\vspace{-2mm}}
\section{Stability set boundary splitting} \gray{}{\vspace{-2mm}}
\label{sec:split}
The 2-parametric description of the ruled surface \eqref{eq:ruled-surface-wt} helps finding smooth parts of the stability set boundary by unwrapping on $w$.

Consider a line \eqref{eq:line-abc} for fixed $w = w_0$ on CRB.
While it is a boundary line, only part of it (or none) is part of the stability set.
Being on a DI-slice with $k_P = k_P(w_0) \doteq k_{P0}$, it is splitted by all other lines \eqref{eq:critical-line}
in the slice, and two special planes.
In parametric form \eqref{eq:ruled-3d} with the third component omitted as it is constant, the line is
\gray{}{\vspace{-1mm}}
\begin{equation} \label{eq:line-pd}
L_0(t) = p_0 + t d_0 =
\begin{bmatrix}
\textstyle -c(w_0) / w_0\\
 0
\end{bmatrix}
+ t \begin{bmatrix}
1/w_0 \vphantom{k_{D0}(w)}\\
w_0 \vphantom{k_{I0}(w)}
\end{bmatrix}.
\gray{}{\vspace{-1mm}}
\end{equation}
Due to the base point choice, $t_0 = 0$ matches intersection with the first critical plane,
and $t_\infty = c(w_0) + w_0 k_{D\infty}$ matches intersection with the second critical plane (if present).

While the whole line \eqref{eq:line-pd} is on the CRB, it describes marginally stable controllers, with
\eqref{eq:charpoly} having a marginal root $j w_0$. Nevertheless, let's consider its topology with respect to \emph{other} roots by one-dimensional D-partition. That is, substitute it into the main boundary equation \eqref{eq:main-jw}, restricting to the line, i.e. solve it versus boundary equations
\eqref{eq:full-ddec-3d}:
\gray{}{\vspace{-1mm}}
\[
\gray{}{\textstyle}
k_D(t) = -c(w_0) + \frac{t}{w_0},\quad k_I(t) = w_0 t, \quad k_P(t) = k_{P0}.
\]
Then RRB gives solution $t_0 = 0$, IRB gives solution $t_\infty$, matching the plane intersection.
The last component in CRB equation leads to discrete set solutions at $w_\ell : k_P(w_\ell) = k_{P0} = k_P(w_0)$.
The first CRB component sets equations
\gray{}{\vspace{-1mm}}
\[ \gray{}{\textstyle}
\gray{}{\vspace{-1mm}}
w_\ell \big(-\frac{c(w_0)}{w_0} + \frac{t}{w_0}\big) - \frac{1}{w_\ell} w_0 t + c(w_\ell) = 0, \quad \ell = 1, ...
\]
These are exactly the equations for intersections with other critical lines \eqref{eq:critical-line} in DI-slice at $k_P = k_{P0}$, with solutions
\gray{}{\vspace{-1mm}}
\begin{equation} \label{eq:t-points}
t_{0,\ell} = w_\ell \frac{w_\ell c(w_0) - w_0 c(w_\ell)}{w_\ell^2 - w_0^2}, \quad w_\ell \neq w_0.
\gray{}{\vspace{-1mm}}
\end{equation}
The ruling line \eqref{eq:line-pd} is splitted by $t_0, t_\infty$ and $t_\ell$ to segments and a ray
(as only positive $t$ match positive $k_I$), to be checked.

The check algorithm is: take any internal point $t^+$ on the segment $(t_i, t_{i+1})$ and corresponding $L(t)$. The characteristic polynomial $G(s, k_D(t^+), k_I(t^+), k_{P0})$ will do have a marginally stable root $j w_0$ (and its complex conjugate), but if \emph{all other roots} are stable, then the segment matches boundary of the stability set (also in \cite{ackermann-kaesbauer2003}).

Following $w_0 \in (0, +\infty)$, all boundary of the stability set is described. Other boundary parts are described by stability regions of PD-controllers in PD-plane, and D-partition of the second critical plane at $k_D = k_{D\infty}$.

Intervals of $w_0$, contributing to the stability set boundary, may be found alongside with
detection of stable $k_P$ intervals, as of \cite{hwang-etal2024}. The critical $k_P$ and corresponding edges of $w$-intervals are coupled by the generator \eqref{eq:kp-generator}.

Let's also note, that if ruling line \eqref{eq:line-pd} have different parametrization starting at alternative base point ($t = 0$)
$k_{D0} = 0, k_{I0} = w_0 c(w_0)$, then crossing points' parameters have more symmetric
expression
\gray{}{\vspace{-1mm}}
\[ \gray{}{\textstyle}
t_\ell = w_0 \frac{w_0 c(w_0) - w_\ell c(w_\ell)}{w_\ell^2 - w_0^2}, \quad w_\ell \neq w_0.
\gray{}{\vspace{-2mm}}
\]
\gray{}{\vspace{-2mm}}
\subsection{Parametric representations of the stability set boundary} \gray{}{\vspace{-2mm}}
The traditional way to visualize stability set boundary is to grid over $k_P$, find all critical lines on DI-slice, determine stable polygons and stack the slices back.

Another way is to grid over $w$, determine all stable segments on a ruling line and plot it in $w-t$ plane.
The latter way is not numerically simpler, as it follows chaining
\[
w \to k_P(w) \to \{w_\ell\} \to \{t_\ell\} \red{\to}{->} \{\text{stable segments}(w)\}.
\]
However, it admits two-dimensional unwrapping of the stability set boundary, alongside
with stable regions in PD-plane and KP-slice at $k_D = k_{D\infty}$, see Fig.~\ref{fig:ex3-unwrap}
with logarithmic on $w$ scale.

\gray{}{\vspace{-2mm}}
\begin{figure}[!h]
\begin{center}
\gray{\includegraphics[width=8.3cm]{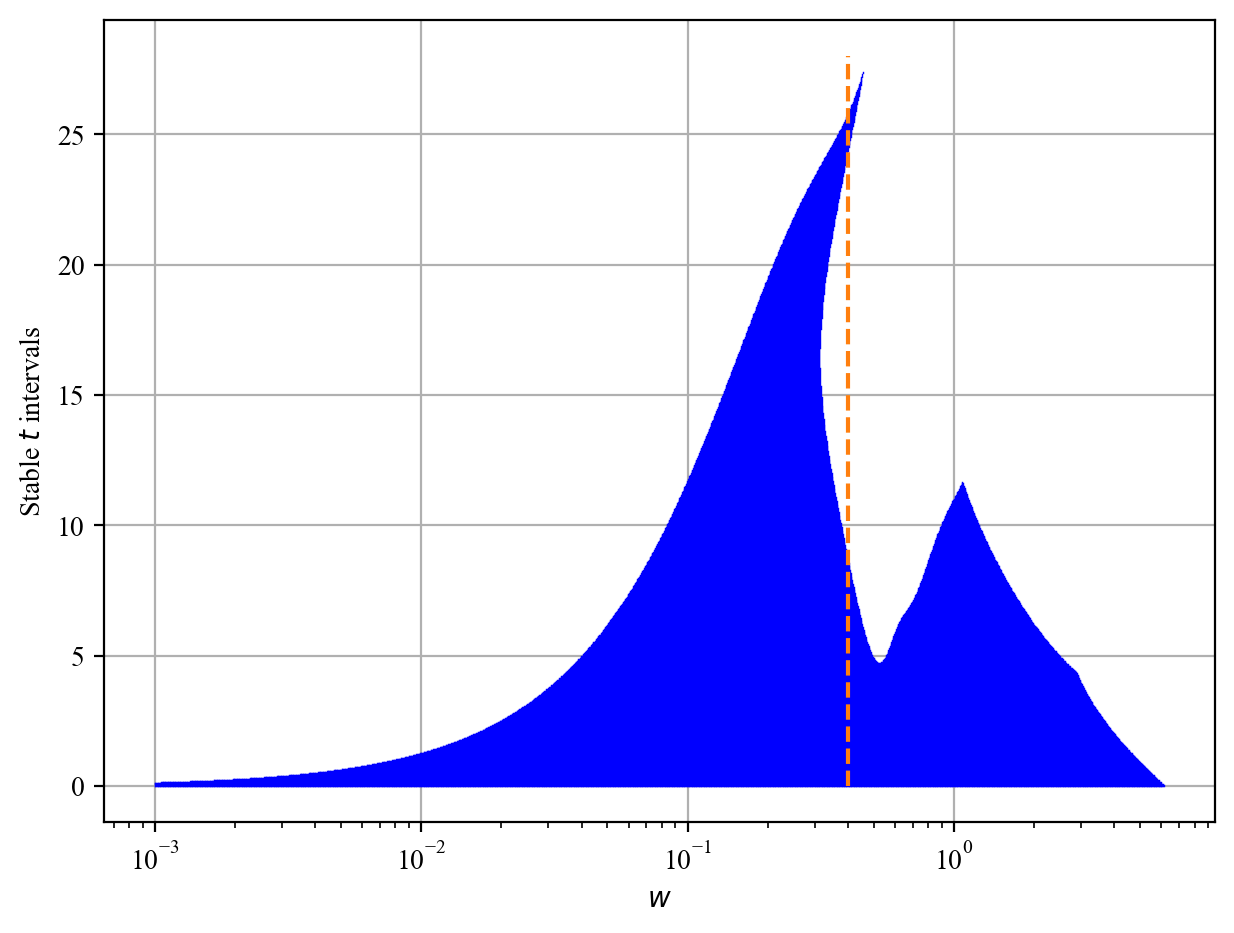}}{\includegraphics[width=5cm]{ack-unwrapping}}
\gray{}{\vspace{-4mm}}
\caption{Unwrapping stability set boundary in $w-t$ space for Example~2. A frequency $w = 0.4$ is selected by dashed line to indicate a case with two segments.}
\label{fig:ex3-unwrap}
\end{center}
\end{figure}
\gray{}{\vspace{-3mm}}
\subsection{Wireframe representation of the stability set} \gray{}{\vspace{-2mm}}
It is convenient to plot only \emph{corners} or extremal lines of the stable polygons. This wireframe (skeleton) representation implies convex linear connection of the wires in horizontal plane, see right parts of Figs.~\ref{fig:ex1}, \ref{fig:ex3}.

\gray{}{\vspace{-1mm}}
\section{Fragility radius} \gray{}{\vspace{-2mm}}
Given a stabilizing PID controller $\overline{k} = (\overline{k}_D, \overline{k}_I, \overline{k}_P)$, its fragility radius describe robustness of the closed-loop system with respect to the controller parameters, e.g. \cite{morarescu-etal2011}:
\gray{}{\vspace{-2mm}}
\[
\begin{array}{c}
r = \min_{G(s, k) \text{ is not stable}}\|k - \overline{k}\|.
\end{array}
\gray{}{\vspace{-1mm}}
\]
As soon the PID stability set boundary being a part of RRB, CRB, IRB, which are planes
ruled surface \eqref{eq:ruled-surface-wt}, the problem of finding a closest unstable controller is decomposed to calculating distances to each of them, as
\begin{align} \nonumber
r = 
\nonumber
\min\{\overline{k}_I, \; |\overline{k}_D - k_{D\infty}|, \; r_{CRB}\}.
\end{align}
($k_{D\infty} \doteq -\infty$ if there is no second critical plane, and $\overline{k}_I > 0$ for stabilizing PID controller.)

The distance to the ruled surface \eqref{eq:ruled-surface-wt} is stated as optimization problem
\gray{}{\vspace{-1mm}}
\[
r_{CRB} = \min_{\begin{array}{c}w > 0, w \neq w_i \\ t \in \R \end{array}} \|A(w) + t B(w) - \overline{k}\|.
\gray{}{\vspace{-1mm}}
\]
It can be explicitly solved on $t$ by projecting $\overline{k}$ to a line \eqref{eq:line-pd} first.
The squared distance to the line is sum of squared distances from the $\overline{k}$'s projection to DI-slice to the line, that is $\frac{w}{\sqrt{w^4 + 1}} (w \overline{k}_D - \frac{1}{w} \overline{k}_I + c(w))$, and vertical squared distance $(k_P(w) - \overline{k}_P)^2$.

Finally optimization over $w$ is performed:
\begin{gather} \nonumber
r_{CRB}^2 = \min_{w > 0} (k_P(w) - \overline{k}_P)^2 + \\[-1mm]
\nonumber
+ \frac{w^2}{w^4 + 1} \Big(w \overline{k}_D - \frac{1}{w} \overline{k}_I + c(w)\Big)^2 \doteq \min_{w > 0} F(w).
\end{gather}
Its minimum is within set of extreme points $w_\ell$, which are solutions of $F'(w) = 0$.
For a delay free plant TF, the function $F(w) = N_F(w)/D_F(w)$ to be minimized is a rational one.
Then its zeros are the real roots of polynomial $N_F'(w) D_{\red{F}{f}}(w) - D_F'(w) N_{\red{F}{f}}(w)$. Taking minimal value among them solves the problem:
\gray{}{\vspace{-1mm}}
\begin{gather*}
r = \min\big\{|\overline{k}_I|, \; |\overline{k}_D - k_{D\infty}|, \\
\gray{}{\textstyle}
\min_\ell \sqrt{(k_P(w_\ell) - \overline{k}_P)^2 +
\frac{1}{w^4 + 1} \big(w_\ell^2 \overline{k}_D -  \overline{k}_I + w_\ell c(w_\ell)\big)^2}\big\}.
\end{gather*}

Note that the formula is not suited for the inverse problem of finding the closest \emph{stable} controller for unstable $\overline{k}$, because it tells distance to the \emph{D-partition boundary} from a point, but not to the \emph{stability set} itself, as the stability set boundary is only small part of D-partition.

For a plant with time-delay, numerical calculation of the fragility radius is due to \cite{morarescu-etal2011}.

\gray{}{\vspace{-2mm}}
\section{Ruled stability boundary for other controllers} \gray{}{\vspace{-2mm}}
\label{sec:non-ideal-pid}
Due to implementation and noise amplification issues, the derivative term in PID is replaced with filtered one, with small $T$, resulting in non-ideal PID controller:
\[
C(s, k_D, k_I, k_P) = k_P + \frac{k_I}{s} + \frac{k_D s}{1 + T s}.
\]
At the first glance, it does not have ruled surface boundary of its stability set due to all parameters being coupled with factors, depending on $s$.
However, in \cite{datta-etal2000} a de-coupling re-parameterization is proposed
\[
\widetilde{k}_P = k_P + k_I T, \quad \widetilde{k}_I = k_I, \quad \widetilde{k}_D = k_D + k_P T.
\]
It converts the problem to case of ideal PID controller with $k = (\widetilde{k}_P, \widetilde{k}_I, \widetilde{k}_D)$, which is to stabilize modified plant
\gray{\[
\frac{N(s)}{(1 + Ts) D(s)}.
\]}{$\frac{N(s)}{(1 + Ts) D(s)}.$}
The stability set boundary is formed by parts of a ruled surface and planes.
After backward linear transform
$\widetilde{k} \to k$,  the stability set boundary still is consists of by parts of a ruled surface and planes\gray{, with constructive description}{}.

The very same idea of coordinate system transform works for the problem of describing full set of $\sigma$-stable PID controllers. The $\sigma$-stability require roots $s_i(k)$ of the closed-loop characteristic polynomial \eqref{eq:charpoly} be not only are stable, but have $\Re s_i < -\sigma$, see \cite{ackermann-kaesbauer2003}.

Next, consider a generic controller in form
\begin{equation} \label{eq:generic-poly}
C(s) = (k_1 P(s) + k_2 Q(s) + k_3 R(s) + S(s)) W_c(s),
\end{equation}
with one part being a TF of a linear subsystem ($W_c(s)$), and polynomial factor, linear on $k$.

\textbf{Conjecture 1:} If with each of polynomials $P, Q, R$ contains either odd or even degrees of $s$ only, then the stabilty set of controller \eqref{eq:generic-poly} is bounded by a ruled surface and planes.

Sketch of the proof: among the three polynomials, at least two have same degree oddity property (say, $P$ and $Q$, without loss of generality).
Dividing characteristic polynomial onto numerators of TFs $W(s)$ and $W_c(s)$, the boundary condition for the closed-loop characteristic
polynomial becomes
\[
k_1 P(jw) + k_2 Q(jw) + k_3 R(jw) + W^{-1}(jw) W_c^{-1}(jw) = 0.
\]
In real-imaginary splitting, the polynomials $P, Q$, alongside with $k_1, k_2$, will vanish
from one of equation, leaving $k_3$ as function of $w$. This works except cases
where all $P, Q, R$ have the same degree oddity, likewise in \cite{nikolaev2014}. In this case, \red{\sout{the}}{the} the equation without $k$ defines a set of critical frequencies, while another
equation describes a set of critical planes, which are degenerate ruled surfaces.

The conjecture is also valid for other controllers, if
there exists a linear transform of parameters,
\gray{$(k_1, k_2, k_3) \to (\widetilde{k}_1, \widetilde{k}_2, \widetilde{k}_3)$,}{}
reducing the controller to form \eqref{eq:generic-poly}.

\textbf{Conjecture 2:} The stabil\gray{i}{}ty set of controller \eqref{eq:generic-poly} is bounded by a ruled surface and planes, if polynomials $P, Q, R$ are linear combination of polynomials $P_1, Q_1, R_1$ satysfying Conjecture 1, coupled by some non-degenerate transform matrix $T \in \R^{3 \times 3}$:
\gray{
\[
\begin{bmatrix}
P(s) \\
Q(s) \\
R(s)
\end{bmatrix} =
T \begin{bmatrix}
P_1(s) \\
Q_1(s) \\
R_1(s)
\end{bmatrix}.
\]}{{\small $\begin{bmatrix}
P(s) \\
Q(s) \\
R(s)
\end{bmatrix} =
T \begin{bmatrix}
P_1(s) \\
Q_1(s) \\
R_1(s)
\end{bmatrix}.$}}

Whether it is the only class of controllers (i.e. the condition of Conjecture 2 is the necessary one), which admit a ruled stability boundary, or how to identify suitable classes of controllers, are open problems.

\gray{}{\vspace{-2mm}}
\section{Examples} \gray{}{\vspace{-2mm}}
Example 1 is from \cite{datta-etal2000}.
\[ \gray{}{\textstyle}
W(s) = \frac{1}{(s+1)^8}.
\]
It has one simple arc-shaped component of ruled surface, cut by the first critical plane $k_I = 0$. It is visualized by wireframe image, see Fig.~\ref{fig:ex1}.
\gray{}{\vspace{-2mm}}
\begin{figure}[!h]
\begin{center}
\includegraphics[width=4.3cm,trim={0mm 0mm 0mm 20mm},clip]{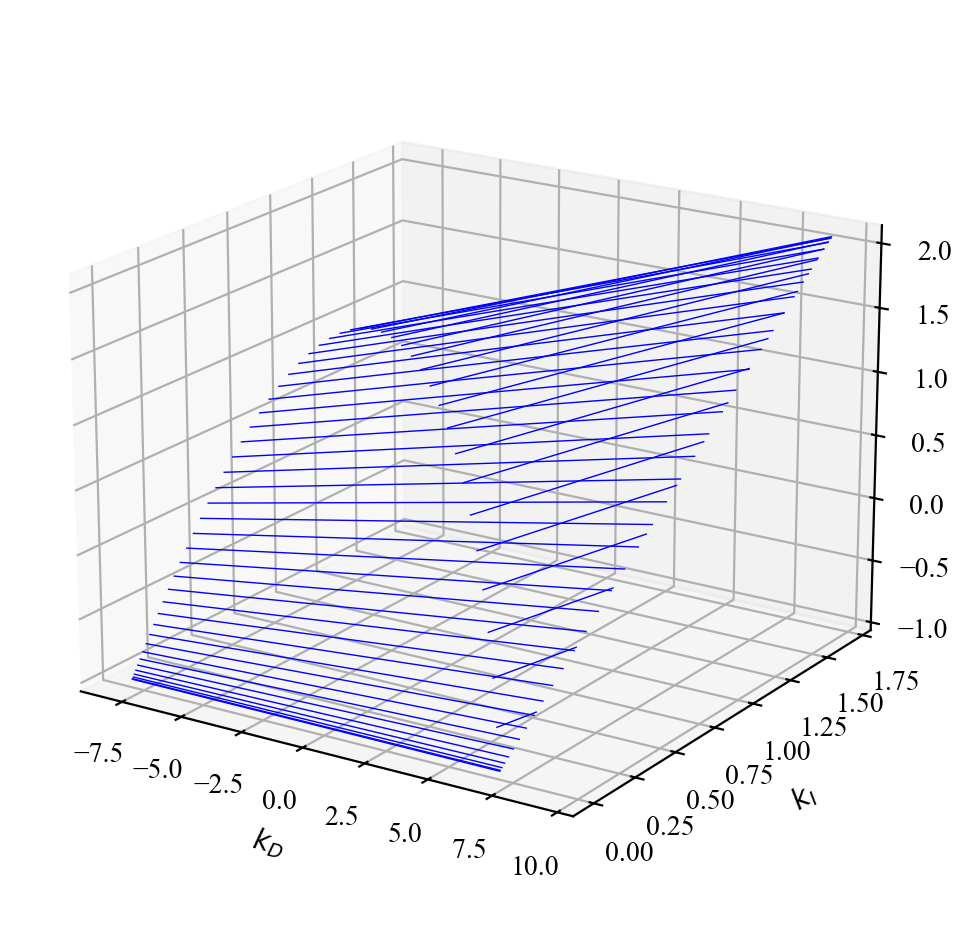}~%
\includegraphics[width=4.3cm,trim={0mm 0mm 0mm 20mm},clip]{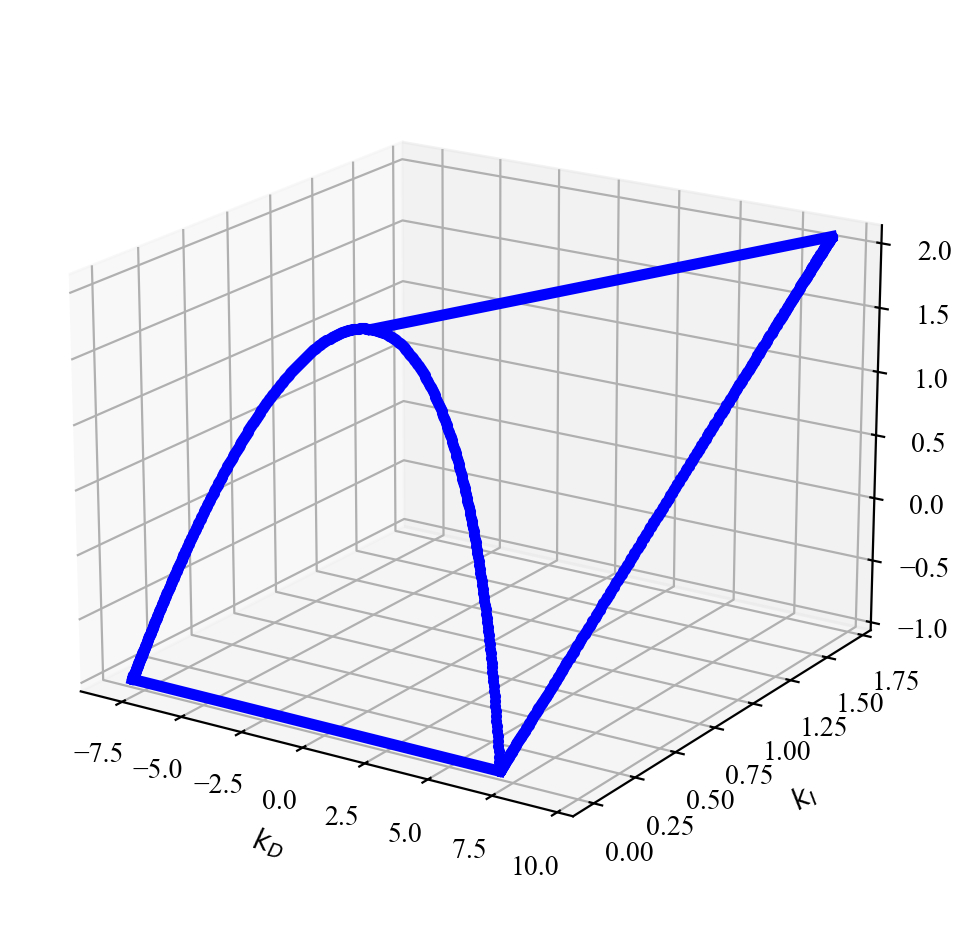} 
\gray{}{\vspace{-4mm}}
\caption{PID stability region, by $w$ gridding and as wireframe for Example 1.}
\label{fig:ex1}
\end{center}
\end{figure}

\gray{}{\vspace{-2mm}}
Example 2 is from \cite{ackermann-etal2002}.
\[ \gray{}{\textstyle}
W(s) = \frac{-s^4 - 7 s^3 - 2s + 1}{(s+1)(s+2)(s+3)(s+4)(s^2 + s + 1)}.
\]
Some $k_P$ slices contain two stable components, see Fig.~\ref{fig:ex3-slice} (right) as example.
The stability set is bounded by a continuous part of ruled surface, cut by the first critical plane.
The stability set boundary by frequency unwrapping is on Fig.~\ref{fig:ex3-unwrap}.
There are some frequencies with two boundary segments, e.g at $w = 0.4$, indicated by dashes.
Full visualization is on Fig.~\ref{fig:ex3}.
\gray{}{\vspace{-2mm}}
\begin{figure}[!h]
\begin{center}
\includegraphics[width=4.3cm,trim={0mm 0mm 0mm 20mm},clip]{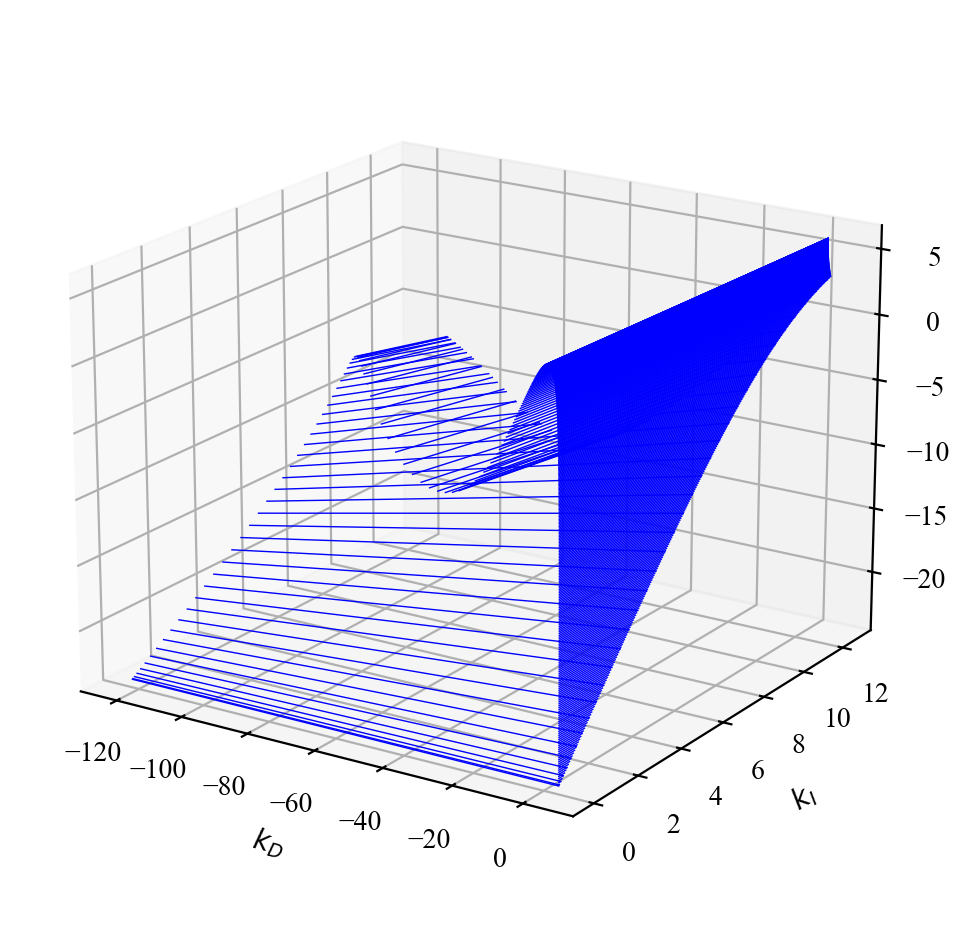}~%
\includegraphics[width=4.3cm,trim={0mm 0mm 0mm 20mm},clip]{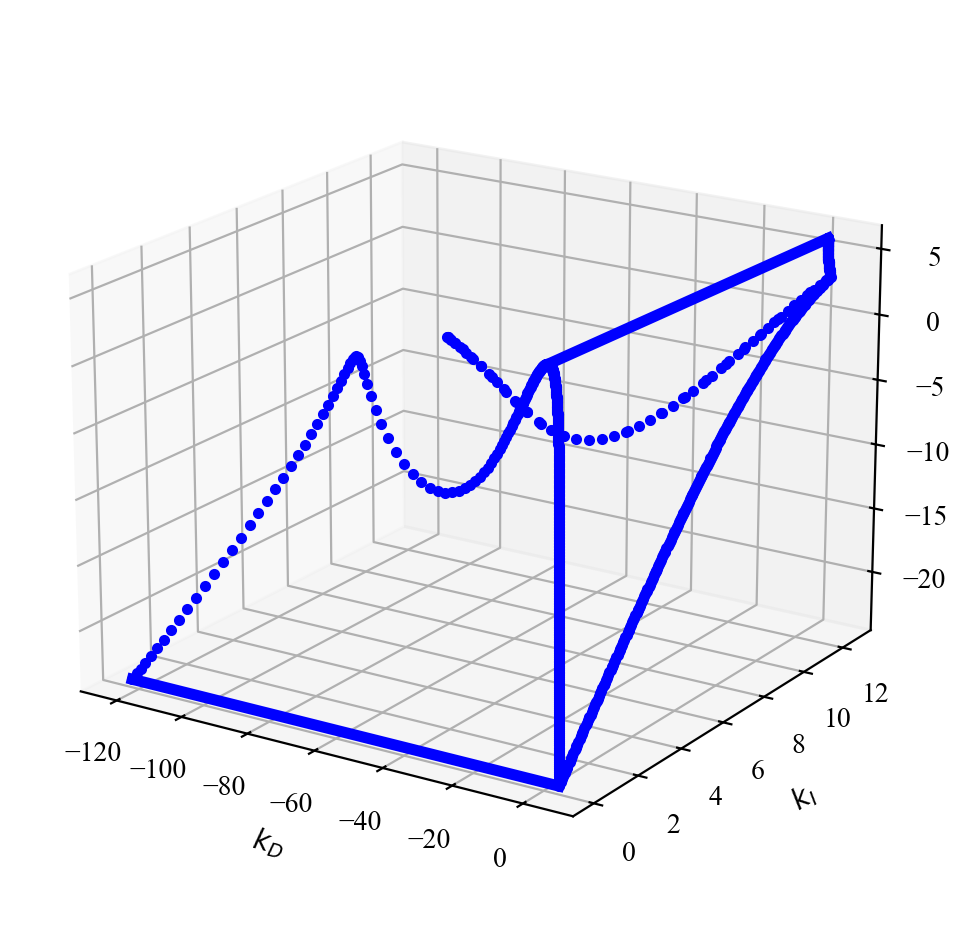} 
\gray{}{\vspace{-4mm}}
\caption{PID stability region, by $w$ gridding and as wireframe for Example 2.}
\label{fig:ex3}
\end{center}
\end{figure}

\gray{}{\vspace{-2mm}}
\section{Conclusion} \gray{}{\vspace{-2mm}}
Boundary of PID controller stability set for linear system in 3D consists of a ruled surface, possibly having few connected parts, and a set of one or two critical planes. The 3D ruled surface geometry is explored, including the case of plant TF with pure imaginary zeros. A sufficient condition for finding stable $k_P$ intervals is provided.
A wireframe visualization of the stability region and its boundary unwrapping in 2D $w-t$ plane are proposed.
At last, solution to fragility radius of a stable PID controller is provided.
The CRB presentation as ruled surface allows to extract individual parts of stability set boundary.

\gray{}{\vspace{-2mm}}
\section*{DECLARATION OF AI-ASSISTANCE} \gray{}{\vspace{-2mm}}
The author used AI-based grammar checking tools only, reviewed proposed changes and take full responsibility for the content of the paper.

\gray{}{\vspace{-2mm}}

\end{document}